\documentclass{daj}

\usepackage{amsthm,amsmath,amscd,amsfonts,amssymb}
\usepackage{mathrsfs,dsfont}

\usepackage{hyperref}
\hypersetup{bookmarksdepth=2}

\numberwithin{equation}{section}
\numberwithin{figure}{section}

\newcommand\R{\mathbb{R}}
 \newcommand\CC{\mathbb{C}}

\newcommand\Z{\mathbb{Z}}
\newcommand\F{\mathcal{F}}

\newcommand\Gam{\Gamma}
\newcommand\lam{\lambda}
\newcommand\Lam{\Lambda}
\newcommand\del{\delta}

\newcommand\sig{\sigma}
\newcommand\Om{\Omega}

\newcommand\1{\mathds{1}}
\newcommand\eps{\varepsilon}

\renewcommand\le{\leqslant}

\renewcommand\leq{\leqslant}
\renewcommand\geq{\geqslant}
\newcommand\sbt{\subset}

\renewcommand\hat{\widehat}
\renewcommand\Re{\operatorname{Re}}

\newcommand{\ft}[1]{\widehat #1}

\newcommand{\mes}{\operatorname{mes}}

\newcommand{\supp}{\operatorname{supp}}

\newcommand{\dist}{\operatorname{dist}}
\newcommand{\sign}{\operatorname{sign}}

\newcommand{\half}{\tfrac{1}{2}}

\theoremstyle{plain}
\newtheorem{thm}{Theorem}[section]
\newtheorem{lem}[thm]{Lemma}

\newtheorem*{claim*}{Claim}

\newcommand{\thmref}[1]{Theorem~\ref{#1}}
\newcommand{\secref}[1]{Section~\ref{#1}}

\newcommand{\lemref}[1]{Lemma~\ref{#1}}
\newcommand{\defref}[1]{Definition~\ref{#1}}

\theoremstyle{definition}
\newtheorem{definition}[thm]{Definition}
\newtheorem*{definition*}{Definition}
\newtheorem*{remarks*}{Remarks}
\newtheorem*{remark*}{Remark}

\newenvironment{enumerate-roman}
{\begin{enumerate}
}
{\end{enumerate}}

\newenvironment{enumerate-alph}
{\begin{enumerate}
}
{\end{enumerate}}

\newenvironment{enumerate-num}
{\begin{enumerate}
}
{\end{enumerate}}

\newenvironment{enumerate-text}
{\begin{enumerate}
}
{\end{enumerate}}

\dajAUTHORdetails{%
  title = {Tiling by translates of a function: results and open problems}, 
  author = {Mihail N. Kolountzakis and Nir Lev},
  plaintextauthor = {Mihail N. Kolountzakis and Nir Lev},
  keywords = {Tiling, translates, spectral gap, uncertainty principle, quasicrystals.},
}   

\dajEDITORdetails{%
   year={2021},
   number={12},
   received={22 September 2020},   
   published={13 September 2021},  
   doi={10.19086/da.28122},       
}   

\begin{document}

\begin{frontmatter}[classification=text]




\author[kolount]{Mihail N. Kolountzakis\thanks{Supported by the Hellenic Foundation for Research and Innovation, Project HFRI-FM17-1733 and by Grant No.\ 4725 of the University of Crete.}}
\author[lev]{Nir Lev\thanks{Supported by ISF Grant No.\ 227/17 and ERC Starting Grant No.\ 713927.}}


\begin{abstract}
We say that a function $f \in L^1(\mathbb{R})$ \emph{tiles} at level $w$ by a discrete translation set $\Lambda \subset \mathbb{R}$, if we have $\sum_{\lambda \in \Lambda} f(x-\lambda)=w$ a.e. In this paper we survey the main results, and prove several new ones, on the structure of tilings of $\mathbb{R}$ by translates of a function. The phenomena discussed include tilings of bounded and of unbounded density, uniform distribution of the translates, periodic and non-periodic tilings, and tilings at level zero. Fourier analysis plays an important role in the proofs. Some open problems are also given.
\end{abstract}
\end{frontmatter}



\section{Introduction} \label{secI1}

Let $f$ be a function in $L^1(\R)$ and let
$\Lam \sbt \R$ be a discrete set. 
We say that \emph{$f$ tiles $\R$
at level $w$} with the translation set $\Lam$,
or that 
\emph{$f+\Lam$ is a tiling of $\R$
at level $w$} (where $w$ is a constant), if  we have
\begin{equation}
\label{eqI1.1}
\sum_{\lambda\in\Lambda}f(x-\lambda)=w\quad\text{a.e.}
\end{equation}
and the series in \eqref{eqI1.1} converges absolutely a.e.

In the same way one can define tiling 
by translates of an $L^1$ function on 
$\R^d$, or  more generally, on any locally compact
 abelian group. The finite abelian groups, 
and in particular the cyclic ones,
are an important class being often considered.

If $f = \1_\Omega$ 
is the indicator function of a set $\Omega$,
and $f + \Lam$ is a tiling at level one, then
this  means
that the translated copies $\Omega+\lam$, $\lam\in\Lam$,
fill the whole space without overlaps up to measure zero. 
To the contrary, for tilings  by a
 general real or complex-valued function
$f$, the translated copies may have 
overlapping supports
 and a wider variety of phenomena may (and do) occur.

In dimension one, translational 
tilings exhibit a stronger structure than in
 higher dimensions, and there are
 interesting questions as to how rigid this
structure must be, e.g.\ how close  the translation set 
$\Lam$ is to being periodic, or to being constructed out
of periodic sets, or, at an even more basic level, 
to being uniformly distributed in $\R$.

The subject has been studied by several authors,
see, in particular, \cite{LM91}, \cite{KL96}, \cite{Kol04}, 
\cite{KL16}, \cite{Liu18}, \cite{Lev20}. 
The aim of this paper is to survey the results 
obtained in earlier works, as well as to prove 
several new results.  At the end of the paper,
some open problems are also given.

In this section we survey the main 
results on the structure of tilings
by translates of a function on $\R$,
 and we state the new results that
will be proved in this paper.

\subsection{Tiling and density}
We say that a set
$\Lambda \sbt \R$ has \emph{bounded density} if
 \begin{equation}
\label{eqI10.1}
 \sup_{x\in\mathbb R} \#(\Lambda\cap[x,x+1))<+\infty.
 \end{equation}
The set $\Lambda$ is said to be
\emph{uniformly distributed} if there is
a number $D(\Lam)$ satisfying
 \begin{equation}
\label{eqI10.2}
\#(\Lambda\cap[x,x+r)) =  D(\Lam) \cdot r  + o(r), \quad r \to +\infty
 \end{equation}
uniformly with respect to $x \in \R$. 
In this case, $D(\Lam)$ is called the
\emph{uniform density} of   $\Lam$.

The following result establishes
a connection between tiling and density:

\begin{thm}[{\cite{KL96}}]
\label{thmKL1}
Let $f+\Lam$ be a tiling at some nonzero level $w$, 
where $f \in L^1(\R)$ and where 
$\Lam \sbt \R$ is a set of bounded density.
Then $f$ has nonzero integral, and $\Lam$ has a uniform
density given by $D(\Lam) = w \cdot (\int f)^{-1}$.
\end{thm}

This was proved in 
\cite[Lemma 2.3]{KL96} for
a weaker notion of 
density\footnote{In \cite{KL96} the density
$d(\Lam)$ of a  set $\Lam \sbt \R$
is defined by
$d(\Lam) := \lim_{r \to +\infty}
\#(\Lambda\cap(-r,r)) / (2r)$.}, but
a minor adjustment to
the proof in fact yields the stronger
statement above
for the  uniform density.
A similar result is true also in $\R^d$.

\subsection{Tiling at level zero}
\label{secTLZ}
It is not known whether
\thmref{thmKL1} has an analog
for tilings of bounded density \emph{at level zero}.
It was conjectured in \cite[p.\ 660]{KL96} 
that if $f+\Lam$ is such a tiling,
then $f$ must have zero integral.
In \cite[Lemma 2.4]{KL96} this was proved 
under the extra assumption that $f$ has compact
support.
 Here we will prove that the conclusion
is true in the general case:

\begin{thm}
\label{thmA5}
Let $f+\Lam$ be a tiling at level zero,
where $f \in L^1(\R)$ and where
$\Lam \sbt \R$ is a nonempty set of bounded density.
Then $f$ must have zero integral.
\end{thm}

Do there exist tilings $f+\Lam$ at level zero 
such that the set $\Lam$ has density zero?
\thmref{thmKL1} does not exclude such
a possibility. In fact, in
 dimensions two and higher it is easy
to exhibit  tilings of this kind.
For instance, in $\R^2$ one may take
$f(x,y) = \varphi(x)\psi(y)$, $\Lam = \Gam \times \{0\}$,
where $\varphi, \psi \in L^1(\R)$ 
and $\varphi + \Gam$ is a tiling of $\R$ at level zero.

We will show, however, that this is \emph{not} the case
in dimension one. To state the result, we recall that a set
 $\Lam \sbt \R$ is said to be
 \emph{relatively dense} if there 
is $r>0$ such that any interval $[x,x+r)$
 contains at least one point from $\Lam$.
We will prove the following:

\begin{thm}
\label{thmA6}
Let $f+\Lam$ be a tiling  at level zero,
where $f \in L^1(\R)$ is  nonzero and 
 $\Lam \sbt \R$ is a nonempty set of bounded density.
Then $\Lam$ must be a  relatively dense set.
\end{thm}

In particular this implies that  $\Lam$ cannot have density zero.

Does it follow from the assumptions in \thmref{thmA6}
that $\Lam$ has a uniform (positive) density $D(\Lam)$\,?
The answer to this question  is not known. 
The problem is nontrivial due to the existence
of  translation sets $\Lam$ 
that admit only tilings at level zero,
so that \thmref{thmKL1} does not apply to these sets:

\begin{thm}[{\cite{Lev20}}]
\label{thmLev20.5}
There exists a nonempty set
 $\Lam \sbt \R$ of bounded density
 which admits tilings $f+\Lam$ with nonzero
$f \in L^1(\R)$, but  any such a tiling
is necessarily a tiling at  level  zero.
\end{thm}

\subsection{Periodic tilings}
We say that   a set $\Lam \sbt \R$ has a \emph{periodic
structure} if it can be represented as 
a disjoint union of finitely
many  arithmetic progressions, namely
\begin{equation}
\label{eqI2.1}
\Lam = \biguplus_{j=1}^{N}
(a_j \Z + b_j)
\end{equation}
where $a_j, b_j$ are real numbers and
$a_j>0$.
The sets $\Lam$  with this structure constitute
 the basic examples of translation sets for
tilings of $\R$. Indeed, one can check  that if
\begin{equation}
f = \1_{[0,a_1]} \ast
\1_{[0,a_2]} \ast \dots \ast
\1_{[0,a_N]}
\end{equation}
and $\Lam$ is given by \eqref{eqI2.1},
then $f + \Lam$ is a tiling at some positive level $w$.

The last example shows that any set
$\Lam$  of the form \eqref{eqI2.1}
admits a tiling by a \emph{compactly supported}
function $f \in L^1(\R)$. A result first proved in
\cite{LM91} and rediscovered in \cite{KL96}
asserts that there are no other translation 
sets $\Lam$ with this property:

\begin{thm}[{\cite{LM91}, \cite{KL96}}]
\label{thmLM91}
Let $f \in L^1(\R)$ be nonzero and have
 compact support.
If $f$ tiles at some level
$w$ with a translation set $\Lambda$ 
of bounded density, then
$\Lambda$ has a periodic structure,
namely,  it  must be of the form \eqref{eqI2.1}.
\end{thm}

The proof of this result is based on Cohen's
characterization of  idempotent measures in locally 
compact abelian groups. 
The group on which Cohen's theorem is used in the
proof  is  the \emph{Bohr compactification} of the real line
(see also \cite[p.\ 25]{Mey70}).

 A discrete set $\Lambda\subset\mathbb R$ is said to have 
\emph{finite local complexity} if $\Lambda$ can be enumerated as a sequence
 $\{\lambda_n\}$, $n\in\Z$, such that $\lambda_n<\lambda_{n+1}$ and the successive differences $\lambda_{n+1}-\lambda_n$ take only finitely many different values. The following result establishes that tilings
of finite local complexity
must be periodic, even if $f$ does not have
 compact support:

\begin{thm}[{\cite{IK13}, \cite{KL16}}]
\label{thmKL16FLC}
Let $\Lambda \sbt \R$ have finite local complexity.
If $f \in L^1(\R)$ is nonzero 
and $f+\Lambda$ is a tiling at some level $w$, then $\Lambda$ 
must be a periodic set, namely, it has the form $\Lam = a\Z + \{b_1, \dots, b_N\}$.
\end{thm}

\subsection{Non-periodic tilings}
The papers \cite{LM91}, \cite{KL96}
leave the following question open: Does
there exist any set $\Lambda \sbt \R$ 
\emph{not} of periodic structure,
which can tile with some
 function $f\in L^1(\R)$ of unbounded support?
Such a set $\Lam$ cannot have finite local complexity
by \thmref{thmKL16FLC}.
We settled this question affirmatively in \cite{KL16}:

\begin{thm}[{\cite{KL16}}]
\label{thmKL16}
There exists a tiling $f+\Lambda$ at level one,
where $f\in L^1(\R)$ and 
$\Lambda \sbt \R$ has bounded density,
but such that  $\Lam$ 
has no periodic structure, 
i.e.\ the set $\Lam$ 
 is not of the form \eqref{eqI2.1}.
\end{thm}

The proof was based on the implicit function method 
due to Kargaev \cite{Kar82}, and it yields a set
$\Lam$ which is a small perturbation of the integers.
The proof moreover allows to choose the function $f$ 
in the Schwartz class. However it
yields a function $f$ satisfying
$\supp(f) = \R$, where 
$\supp(f)$ is the closed support of $f$
(the smallest closed set such that $f$ vanishes
a.e.\ on its complement).

In this paper we will prove a stronger version
of Theorem \ref{thmKL16}, which establishes the existence
of non-periodic tilings $f+\Lam$ of bounded density,
such that $f$ has ``sparse'' support.
Precisely, we will show that the
support (which must be unbounded, due to
\thmref{thmLM91}) can be localized 
inside  any given set $\Om \sbt \R$
which contains arbitrarily long intervals:

\begin{thm}
  \label{thmA1}
There is a discrete set $\Lambda \sbt \R$ of bounded density,
 but which is not of the form \eqref{eqI2.1}, with the following
property: given any  scalar $w$, and any set
$\Om \sbt \R$ which contains arbitrarily long intervals,
one can find  a nonzero 
$f \in L^1(\mathbb R)$, $\supp(f) \sbt \Om$,
such that $f+\Lambda$ is a tiling at level $w$.
\end{thm}

We note that the set $\Om$ in this theorem may
 be chosen very sparse, for example, one may take
$\Om = \bigcup_{j=1}^{\infty} [\tau_j, \tau_j + j]$,
where the numbers $\tau_j$   grow arbitrarily fast.

\thmref{thmA1} implies in particular
 the existence of  non-periodic tilings
by a function $f$ such that
$\supp(f) \sbt [0, +\infty)$, i.e.\ the
support is \emph{bounded from below}
(while it cannot be
bounded from both above and below,
due to \thmref{thmLM91}).

The proof of \thmref{thmA1} relies on a recent result
due to Kurasov and Sarnak \cite{KS20},
who constructed a new type
of \emph{crystalline measures} on $\R$.
These are pure point
measures with discrete closed support,
whose Fourier transform 
is a measure of the same type.

\subsection{Tilings of unbounded density}
\label{secTUB}
It seems that very little attention
has been paid to 
tilings which are \emph{not} of bounded density.
In fact, we are not aware of any 
example of such a tiling  in the literature.
In \cite[Example 7.1]{KL96} 
some examples are given 
in a more general context,
where the points of the
translation set $\Lam$
are endowed with
nonnegative integer multiplicities
(so that $\Lam$ is actually not a set,
but a multi-set).

We will prove that there exist tilings 
of unbounded density in the proper sense.
Let us say that a set $\Lam \sbt \R$ has 
\emph{tempered growth} if there is $N$ such that
\begin{equation}
\label{eqA3.16}
\# (\Lam \cap (-r,r)) = O(r^N), \quad r \to +\infty.
\end{equation}

\begin{thm}
  \label{thmA3}
There is a set $\Lambda \sbt \R$ which is
 not of bounded density but has tempered growth,
 with the following
property: given any scalar $w$ one can find
a nonzero function $f$ in the Schwartz class, 
 such that $f+\Lambda$ is a tiling at level $w$.
\end{thm}

Moreover, in our example the set $\Lam$ is
contained in a small 
neighborhood of the integers. The
construction is done using a variant of Kargaev's 
implicit function method.

It follows from \cite[Lemma 2.1]{KL96}
that the function $f$ in \thmref{thmA3} must
change sign, i.e.\ $f$ cannot be
chosen nonnegative. This indicates  that
 cancellations are playing a decisive role in the 
tiling. We can even prove a stronger claim:

\begin{thm}
  \label{thmA9}
Let $f+\Lam$ be a tiling at some level $w$, 
where the set $\Lambda \sbt \R$ is  not of bounded density 
but has tempered growth, and where
 $f$ is a function in the Schwartz class.
Then $f$ must have zero integral.
\end{thm}

It may seem counter-intuitive at first glance that 
one can tile $\R$ at a nonzero level $w$ 
by translates of a Schwartz function  $f$ whose integral is zero.

\subsection{Organization of the paper}
The rest of the paper is organized as follows.

In \secref{secF1} some preliminary background
is given, and Fourier analytic conditions for tiling 
are discussed.

 In \secref{secL1} we prove that
 if $f+\Lam$ is a tiling at level zero,
where $f \in L^1(\R)$ is a  nonzero function and 
 $\Lam \sbt \R$ is a nonempty set of bounded density,
then $f$ has zero integral
(\thmref{thmA5}) and 
 $\Lam$ is  relatively dense  (\thmref{thmA6}).

In \secref{secSM1} we establish the existence
of non-periodic tilings $f+\Lam$ of bounded density,
such that the function $f$ has ``sparse'' support
(\thmref{thmA1}).

In \secref{secUB1} we construct
tilings $f+\Lam$ of unbounded density,
such that $\Lam$ has  tempered growth 
and  $f$ is in the Schwartz  class
(\thmref{thmA3}). We also show
that in any such a tiling, the function $f$ must have
zero integral (\thmref{thmA9}).

In the last section, \secref{secOP1},
we pose some open problems.


\section{Preliminaries. Fourier analytic conditions for tiling.}
 \label{secF1}

It is well-known that in
 the  study of  translational tilings, 
Fourier analysis plays an important role,
see e.g.\ \cite{Kol04}.
If $f \in L^1(\R)$ then its Fourier transform
is defined by
\begin{equation}
\hat{f}(t) = \int_{\R} f(x) \, e^{-2\pi i t x} \, dx.
\end{equation}

If $\alpha$ is a tempered distribution on $\R$,
then $\alpha(\varphi)$ denotes the action 
of $\alpha$ on a Schwartz function  $\varphi$.
 The Fourier transform $\ft\alpha$ is defined by 
$\ft{\alpha}(\varphi) = \alpha(\ft{\varphi})$.

If $\alpha$ is a tempered distribution on $\R$, and if
$\varphi$ is a Schwartz function, then
the convolution $\alpha \ast \varphi$ is 
a tempered distribution whose Fourier transform
is $\ft{\alpha} \cdot \ft{\varphi}$.

We use $\supp(\alpha)$ to denote the closed support
of a tempered distribution $\alpha$.

If $f \in L^1(\R)$ then $\supp(f)$ is the
smallest closed set such that $f$ vanishes
a.e.\ on its complement. This set coincides 
with the support of $f$ in the distributional sense.

For more details on distribution theory we refer
the reader to  \cite{Rud91}.

\subsection{Necessary conditions for tiling}
For a discrete set $\Lambda \sbt \mathbb{R}$
 we define the measure 
\begin{equation}
\label{eqI5.4}
\delta_\Lambda:=\sum_{\lambda\in\Lambda}\delta_\lambda.
\end{equation}
The tiling condition 
\eqref{eqI1.1}  can then be restated  as 
\begin{equation}
\label{P2.2.11}
f \ast \delta_\Lam = w \quad \text{a.e.}
\end{equation}
If $\Lambda$ has bounded density 
then the measure $\delta_\Lambda$ is a tempered
 distribution on $\R$. So, at least formally, taking
the Fourier transform of both 
sides of \eqref{P2.2.11} yields
\begin{equation}
\label{P2.2.12}
\ft{f} \cdot \ft{\delta}_\Lam = w \cdot \delta_0.
\end{equation}

If $f$ is a Schwartz function,
then condition \eqref{P2.2.12}  makes sense
and it is equivalent to \eqref{P2.2.11}.
To the contrary, for an arbitrary function $f \in L^1(\R)$ 
(not assumed to be in the Schwartz class)
 the product 
$\ft{f} \cdot \ft{\delta}_\Lam$
is not well-defined,
and the condition \eqref{P2.2.12}
 can only be interpreted as a heuristic principle.

The following result, inspired by 
the heuristic condition \eqref{P2.2.12},
 provides a necessary condition for tiling
which holds for any $f \in L^1(\R)$.

\begin{thm}[{\cite{KL16}}]
\label{thm4.5.1}
Let $f\in L^1(\mathbb R)$, and $\Lambda\sbt\R$
 be a discrete set of bounded density. If $f+\Lambda$ is a tiling
at some level $w$, then
\begin{equation}
\label{eqI3.1}
\supp(\hat\delta_\Lambda) \setminus \{0\}
\subset \{\hat f=0\}.
\end{equation}
\end{thm}

The proof of this result is based on Wiener's tauberian theorem.
In the earlier works \cite{KL96}, \cite{Kol00a}, \cite{Kol00b}
the result was proved under various extra assumptions.

\subsection{Sufficient conditions for tiling}
One may ask whether 
\thmref{thm4.5.1}  admits a converse, i.e.\ if the condition
\eqref{eqI3.1} implies that $f+\Lambda$ is a tiling
at some level $w$. 
If the distribution $\ft{\delta}_\Lam$ happens to
be a \emph{measure} on $\R$, then 
the product $\ft{f} \cdot \ft{\delta}_\Lam$
is a well-defined measure and the
condition \eqref{P2.2.12} makes sense.
In this case, the two conditions \eqref{P2.2.12} and \eqref{eqI3.1}
are equivalent, and can be shown to imply that $f+\Lam$ is a tiling:

\begin{thm}\label{thm4.5.2}
Let $\Lambda \sbt \R$ have bounded density and suppose that 
$\ft{\delta}_\Lambda$ is a locally finite measure.
If $f\in L^1(\R)$ satisfies \eqref{eqI3.1} then
 $f+\Lambda$ is a tiling at level
$w =  \ft{\delta}_\Lam(\{0\}) \cdot \int f$.
\end{thm}

A proof of this result is given below.
In \cite{KL96} the result was proved 
under the extra assumption that $\ft{f}$ is 
a smooth function.

As an example, \thmref{thm4.5.2} 
applies if the set $\Lambda$ has a periodic structure,
i.e.\ if it has the form \eqref{eqI2.1}. In this case
$\ft{\delta}_\Lambda$ is a (pure point) 
measure by Poisson's summation formula,
and the condition \eqref{eqI3.1} is
both necessary and sufficient for
a function  $f\in L^1(\R)$ to tile
at some level $w$ with the translation set $\Lam$.

\thmref{thm4.5.2} leaves open, though, the question as to
whether  the condition \eqref{eqI3.1} is sufficient
for tiling in the general case, i.e.\ for an arbitrary
 set $\Lam \sbt \R$
 of bounded density. This question was addressed 
recently in  \cite{Lev20} where it
was answered in the negative:

\begin{thm}[{\cite{Lev20}}]
\label{thmI8.10}
There is a set $\Lambda \subset \R$  of bounded density
and a function $f \in L^1(\mathbb{R})$,
such that \eqref{eqI3.1} is satisfied however $f + \Lam$ is not a tiling
at any level.
\end{thm}

The proof of this result is based on the relation of the problem to
Malliavin's \emph{non-spectral synthesis} example.

\thmref{thmI8.10} implies that a converse
to \thmref{thm4.5.1} can only be valid under certain
extra assumptions.  One example of such a converse
is given by \thmref{thm4.5.2}. Another example 
is the following result:

\begin{thm}\label{thm4.5.8}
Let $f\in L^1(\mathbb R)$, and let
$\Lambda \sbt \R$ be a discrete set of bounded density.
 If the set
$\supp(\hat\delta_\Lambda) \setminus \{0\}$ is closed
and is disjoint from $\supp(\ft{f})$, 
then  $f+\Lambda$ is a tiling at some level $w$.
\end{thm}

We will not use this result in the paper
and we do not include its proof. For other results of similar type
see \cite[Theorem 3]{Kol00a}, \cite[Theorem 5]{Kol00b}.

\subsection{Translation-bounded measures}
A measure $\mu$ on $\R$  satisfying the condition
\begin{equation}
\sup_{x \in \R} |\mu|([x,x+1)) < \infty
\end{equation}
is said to be \emph{translation-bounded}.
If a measure $\mu$ is translation-bounded,
then it is a tempered distribution. 
If $\mu$  is a translation-bounded measure  on $\R$, and if $\nu$
is a finite measure  on $\R$, then the convolution $\mu \ast \nu$
is a translation-bounded measure. 

\begin{thm}
\label{thmC2}
Let $\mu$ be a translation-bounded measure on $\R$,
and suppose that  $\ft{\mu}$ is a locally finite measure.
If $\nu$ is a finite measure on $\R$, then the Fourier 
transform of the convolution $\mu \ast \nu$
 is the measure $\ft{\mu} \cdot \ft{\nu}$.
\end{thm}

In particular, let  $\Lambda \sbt \R$ be a discrete set
 of bounded density, and let $f \in L^1(\R)$.
Then the measure $\delta_\Lam$ is 
 translation-bounded, and the convolution
$f \ast \del_\Lam$ is the sum of the
series  $\sum_{\lambda\in\Lambda}f(x-\lambda)$
which converges absolutely a.e., see
\cite[Lemma 2.2]{KL96}.
If the  distribution
$\ft{\delta}_\Lambda$ is assumed to be a locally finite measure,
then using \thmref{thmC2} we obtain that
the three conditions
\eqref{P2.2.11}, \eqref{P2.2.12} and \eqref{eqI3.1} are
equivalent. We thus
see that \thmref{thm4.5.2} is a consequence
of \thmref{thmC2}.

\thmref{thmC2} should be known to experts but we could not
find a proof in the literature,
 so we include one below for completeness.

\begin{proof}[Proof of \thmref{thmC2}]
We will use $\sig$ to denote the measure $\ft{\mu}$.
The assertion of the theorem means that 
if $\psi$ is a Schwartz function with compact support then
\begin{equation}
  \label{eqP1.10}
\int_{\R}  \ft{\psi(x)} \, d(\mu \ast \nu)(x) = \int_{\R} \psi(t) \, \ft{\nu}(t) \, d\sig(t).
\end{equation}

Let $\chi$ be a Schwartz function whose Fourier transform
$\ft{\chi}$ is non-negative, has compact support,
 $\int \ft{\chi}(t) dt =1$, and for each $\eps > 0$ let
$\chi_\eps(x) := \chi( \eps x)$. Let
$h_\eps := (\psi \cdot \ft{\nu}) \ast \ft{\chi}_\eps$,
then $h_\eps$ is an infinitely smooth function with compact support.
As $\eps \to 0$, the function $h_\eps$ remains supported on
a certain  interval $I=[a,b]$ that does not depend on $\eps$, and
$h_\eps$ converges to 
$\psi \cdot \ft{\nu}$ uniformly on  $I$. Hence
\begin{equation}
  \label{eqP1.3}
\lim_{\eps \to 0} \ft{\mu}(h_\eps) = 
\lim_{\eps \to 0}
\int_{\R} h_\eps(t)  d\sig(t) =
\int_{\R} \psi(t) \, \ft{\nu}(t) \,  d\sig(t).
\end{equation}

The function $\psi$ is the Fourier transform of some function
$\varphi$ in the Schwartz class. 
Let $g_\eps := (\varphi \ast \nu) \cdot \chi_\eps$,
then $g_\eps$ is a smooth function in $L^1(\R)$ and
we have $\ft{g}_\eps = h_\eps$.
Since $h_\eps$ belongs to the Schwartz 
space, the same is true for $g_\eps$, and it follows that
\begin{equation}
  \label{eqP1.7}
\ft{\mu}(h_\eps) = \mu(\ft{h}_\eps)
= \int_{\R} {g_\eps}(-x) \, d\mu(x)
= \int_{\R}  (\varphi \ast \nu)(-x) \, \chi_\eps(-x)\, d\mu(x).
\end{equation}
We observe that $|\chi_\eps(-x)| \leq 1$ and
$\chi_\eps(-x) \to 1$ pointwise as $\eps \to 0$.
We now wish to apply the dominated convergence
theorem to \eqref{eqP1.7}. Using Fubini's theorem 
we obtain 
\begin{equation}
  \label{eqP1.23}
 \int_{\R}  (|\varphi| \ast |\nu|)(-x) \, |d\mu|(x)
= \int_{\R}   |\varphi(-x)| \, d(|\mu| \ast |\nu|)(x).
\end{equation}
The function $\varphi$ has fast decay being in the
Schwartz class, while
$|\mu| \ast |\nu|$ is a translation-bounded measure,
hence the integral in \eqref{eqP1.23} converges.
We may therefore apply to \eqref{eqP1.7}
the dominated convergence theorem, which yields
\begin{equation}
  \label{eqP1.8}
\lim_{\eps \to 0} \ft{\mu}(h_\eps)
= \int_{\R}  (\varphi \ast \nu)(-x) \, d\mu(x)
= \int_{\R}   \varphi (-x) \, d(\mu \ast \nu)(x).
\end{equation}
But $\varphi(-x) = \ft{\psi}(x)$, so we see
that \eqref{eqP1.10} follows from
 \eqref{eqP1.3} and \eqref{eqP1.8} as needed.
\end{proof}


\section{Tiling at level zero} 
\label{secL1}

 In this section we prove that
 if $f+\Lam$ is a tiling at level zero,
where $f \in L^1(\R)$ is  nonzero and 
 $\Lam \sbt \R$ is  nonempty and has bounded density,
then $f$ has zero integral
(\thmref{thmA5}) and 
 $\Lam$ is a relatively dense set (\thmref{thmA6}).

\subsection{The function \texorpdfstring{$f$}{f} has zero integral}
\label{subsecZI}
We begin with \thmref{thmA5}, for which we 
give two proofs.  The first proof 
is Fourier analytic, and is based on the following result:

\begin{thm}[{\cite{KL16}}]
\label{thm5.1}
Let $f\in L^1(\mathbb R)$ have nonzero integral. If  $\Lambda
\sbt \R$
 is a set of bounded density  and  $f+\Lambda$ is a tiling
at some level $w$, then there is $a>0$ such that
$\hat\delta_\Lambda=c\cdot\delta_0$
in $(-a,a)$, where $c=w\cdot(\int f)^{-1}$.
\end{thm}

This result follows from the proof of \cite[Theorem 4.1]{KL16}
and it is stated in that paper as a remark on p.\ 4598.
The proof of \thmref{thm5.1} is based on Wiener's tauberian theorem.

\begin{proof}[First proof of \thmref{thmA5}]
Let $f+\Lambda$ be a tiling
at level $w=0$, and suppose to the contrary that $f$ has nonzero integral.
Then by \thmref{thm5.1} there is $a>0$ such that
$\hat\delta_\Lambda=c\cdot\delta_0$ in $(-a,a)$,
where $c=w\cdot(\int f)^{-1}$. Since $w=0$
this means that  $\ft{\delta}_\Lam$
vanishes in a neighborhood $(-a,a)$ of the origin. 
We will show that this cannot happen.

Indeed, choose a Schwartz function 
$\varphi$ such that $\supp(\varphi)\subset (-a,a)$
and $\hat\varphi>0$. Then we have
$\ft{\del}_\Lam(\varphi)  = 0$ since
 $\ft{\delta}_\Lam$ vanishes in a neighborhood 
of $\supp(\varphi)$. On the other hand,
\[
\ft{\del}_\Lam(\varphi) = 
\del_\Lam(\ft{\varphi}) =
\sum_{\lam\in\Lam}  \ft{\varphi}(\lam) > 0,
\]
since $\Lam$ is  nonempty and 
 $\hat\varphi$ is everywhere positive.
We thus arrive at a contradiction.
\end{proof}

As a remark, we record here the following observation:

\begin{lem}
\label{lem5.2}
Let $\Lambda \sbt \R$ be a nonempty set of 
bounded density,   and suppose
that there is $a>0$ such that
$\hat\delta_\Lambda=c\cdot\delta_0$
in $(-a,a)$. Then $c>0$, and $\Lam$ has a uniform
density given by $D(\Lam) = c$. 
\end{lem}

\begin{proof}
Let $\varphi$ be a Schwartz function 
such that $\varphi > 0$, $\int \varphi = 1$ and
$\supp(\ft{\varphi})\subset (-a,a)$.
Then we have  $\ft{\varphi} \cdot \ft{\del}_\Lam = c \cdot \del_0$
and hence $\varphi  + \Lam$ is a tiling at level $c$.
Since $\varphi$ is a positive function and
$\Lambda$ is nonempty,
 the tiling level $c$ must also be  positive.
Finally, we obtain from \thmref{thmKL1} that
the set $\Lam$ has a uniform density given by $D(\Lam) = c$. 
\end{proof}

Next we give our
 second proof  of \thmref{thmA5}. This proof, which
 does not involve Fourier analysis, is based on the
following  result due to Ruzsa and Sz\'{e}kely \cite{RS83}.

\begin{thm}[{\cite{RS83}}]
\label{thmRS83}
Let $f \in L^1(\R)$ be  real-valued with $\int f > 0$.
Then there is a  nonnegative (nonzero) $g \in L^1(\R)$ 
such that the convolution $f \ast g$ is nonnegative a.e.
\end{thm}

In fact this is proved in \cite{RS83} not only
for functions on $\R$, but on a wider class of locally 
compact abelian groups,
which in particular includes also $\R^d$ for every $d \geq 1$.

\begin{proof}[Second proof of \thmref{thmA5}]
Let $f+\Lambda$ be a tiling
at level zero, and suppose to the contrary that 
$\int f$ is nonzero. By replacing 
$f(x)$ with the function
$\Re\big\{(\int f)^{-1}  f(x)\big\}$, 
we may assume that $f$
is real-valued and that $\int f >0$.
So we can use \thmref{thmRS83} to find
 a nonnegative, nonzero $g \in L^1(\R)$
such that $f \ast g$ is nonnegative a.e.
Notice that also the function $f \ast g$ is nonzero, since
we have $\int (f \ast g) = (\int f)(\int g) > 0$.

On the other hand we have
$f \ast \del_\Lam = 0$ a.e., and the measure $\delta_\Lam$ is 
 translation-bounded since $\Lam$ has
bounded density. Using
Fubini's theorem this implies that
\[
(f \ast g) \ast \del_\Lam = 
g \ast (f \ast \del_\Lam) = g \ast 0 = 0 \quad \text{a.e.,}
\]
hence $(f \ast g) + \Lam$
is a tiling at level zero. But this is a contradiction,
since $f \ast g$ is a  nonnegative, nonzero
function in $L^1(\R)$, and  $\Lam$ is a
nonempty set.
\end{proof}

\subsection{The set \texorpdfstring{$\Lambda$}{Lambda} is relatively dense}
We now  move on to the proof of 
 \thmref{thmA6}. First we note that this
theorem \emph{cannot} be deduced based on 
\lemref{lem5.2} above, since 
there exist tilings  $f+\Lam$ 
at level zero such that
$\ft{\del}_\Lam$ is \emph{not} 
a scalar multiple of $\del_0$ in any
neighborhood of the origin; see
\cite[Section 5]{Lev20} for a
 construction of such tilings.

We fix the following terminology: given 
a sequence of measures $\{\mu_n\}$ on $\R$,
we say that the sequence  is 
\emph{uniformly translation-bounded} if there is 
a constant $M>0$ not depending on $n$, such that
$\sup_{x \in \R} |\mu_n|([x,x+1)) \leq M$
 for every $n$.

If $\{\mu_n\}$ is a uniformly translation-bounded sequence of measures
 on $\R$, then we say  that $\mu_n$ \emph{converges vaguely} to a 
measure $\mu$ if for
every continuous, compactly supported function $\varphi$ we have
$\int \varphi \, d\mu_n \to \int \varphi \, d\mu$
(see, for instance, \cite[Section 7.3]{Fol99}).
 In this case, the vague limit $\mu$ 
must also be a translation-bounded measure.
 For a uniformly translation-bounded 
sequence of measures  $\{\mu_n\}$  to converge vaguely,
it is necessary and sufficient that $\{\mu_n\}$  converge
 in the space of tempered distributions.
From any uniformly translation-bounded sequence of 
measures $\{\mu_n\}$  one can extract
a vaguely convergent subsequence $\{\mu_{n_j}\}$.

\begin{proof}[Proof of \thmref{thmA6}]
Let $f+\Lam$ be a tiling  at level zero, and
suppose to the contrary that $\Lam$ is not relatively dense. 
Then for each $r>0$ one can find an open interval $I(r)$
 of length $r$ which is disjoint from $\Lam$.
 By translating  the  interval $I(r)$ we may assume 
that one of its endpoints lies in $\Lam$
(since the set $\Lam$  is nonempty).
Then for arbitrarily large values of 
$r$ the right endpoint of 
$I(r)$ is in $\Lam$, or for arbitrarily large values of 
$r$ the left endpoint is in $\Lam$.
We will assume that the former is the case
(the latter case can be treated similarly).
It follows that there exist two sequences $r_j \to +\infty$ and 
$\lam_j \in \Lam$,  such that for each $j$ the interval
$I_j := (\lam_j - r_j, \lam_j)$ does
not intersect $\Lam$.

Define $\Lam_j := \Lam - \lam_j$. Since $\Lam$ has bounded density,
the measures $\delta_{\Lam_j}$ are uniformly translation-bounded.
By passing to a subsequence if necessary, we may  assume
that $\delta_{\Lam_j}$ converges vaguely to some 
(also translation-bounded) measure $\mu$ on $\R$.
The vague limit $\mu$ is supported on $[0, +\infty)$, since 
$\delta_{\Lam_j}$ vanishes on
$(-r_j, 0)$. Moreover,  each measure $\delta_{\Lam_j}$
is  positive and has a unit mass at the origin, which
implies that $\mu$ is also positive
and has an atom at the origin, of mass at least one.
In particular the measure $\mu$ is nonzero.

On the other hand, since $f+\Lam$ is a tiling (at level zero),
then by \thmref{thm4.5.1} the support of
$\ft{\delta}_\Lambda$ is contained in $\{\hat f=0\} \cup \{0\}$.
Since $\ft{f}$ is a nonzero continuous function, it follows that
there is an open interval $(a,b)$ on which $\ft{\delta}_\Lam$ vanishes. 
In turn this implies that
 $\ft{\delta}_{\Lam_j}$ also vanishes on $(a,b)$
for each $j$. But as the sequence $\ft{\delta}_{\Lam_j}$ converges to
$\ft{\mu}$ in the distributional sense, we obtain
that $\ft{\mu}$ vanishes on $(a,b)$ as well. 

We conclude that $\mu$ is a nonzero, translation-bounded, positive 
 measure on $\R$, supported on $[0, +\infty)$ and
 whose Fourier transform 
$\ft{\mu}$ vanishes on some
open interval $(a,b)$. But 
this contradicts classical results
on boundary values of functions analytic 
in the upper half-plane. To be more concrete:
choose two Schwartz functions  $\varphi, \psi$
such that $\varphi > 0$,
$\supp(\ft{\varphi}) \sbt (-\del,\del)$, $\psi$ is nonnegative,
$\psi$ is
supported on $[0,+\infty)$ and $\int \psi = 1$.
Then 
$h := (\mu \cdot \varphi) \ast \psi$
is a nonzero function belonging to $L^1(\R)$, $\supp(h) \sbt
[0,+\infty)$, and the Fourier transform
$\ft{h} = (\ft{\mu} \ast \ft{\varphi}) \cdot  \ft{\psi}$
is also in $L^1(\R)$ and $\ft{h}$ vanishes on some nonempty
open interval provided that  $\delta$ is small enough.
This contradicts the uniqueness theorem for
functions in the Hardy space $H^1$, see
e.g.\ \cite[Chapter II]{Gar07}.
\end{proof}


\section{Non-periodic tilings by  functions with sparse support}
\label{secSM1}

In this section we establish the existence
of non-periodic tilings $f+\Lam$ of bounded density,
such that the function $f$ has ``sparse'' support
(\thmref{thmA1}).

Recall that
the first example of a tiling $f+\Lam$ such that
$f\in L^1(\R)$, $\Lambda \sbt \R$ has bounded density,
but the set $\Lam$ does \emph{not} have a periodic structure, 
was given in \cite{KL16}.
The proof was based on the implicit function method 
due to Kargaev \cite{Kar82}, and it yields a set
$\Lam$ which is a small perturbation of the integers.
However, in this example the function $f$ is
\emph{analytic}, which implies that  $\supp(f) = \R$.

In order to prove \thmref{thmA1} we will use a different approach,
which is based on two main ingredients. The first one is a
recent construction from \cite{KS20}
of a  new type of \emph{crystalline measures} on $\R$.
The second main ingredient is a result concerning
interpolation of discrete functions by continuous ones with
a sparse spectrum.

\subsection{Crystalline measures}
A tempered distribution $\mu$ on $\R$ satisfying
\begin{equation}
  \label{eqI1.10}
\mu = \sum_{\lam\in \Lam} a_\lam \del_\lam,
\quad
\ft{\mu} = \sum_{s \in S} b_s \del_s
\end{equation}
(i.e.\ both $\mu$ and $\ft{\mu}$ are pure point measures),
where $\Lam$ and $S$ are discrete, closed sets in $\R$,
is called a \emph{crystalline measure} \cite{Mey16}.
This notion plays a role in the 
mathematical theory of quasicrystals, 
i.e.\ atomic arrangements having
a discrete diffraction pattern.

The classical example of a crystalline measure
is $\mu = \delta_\Z$, which satisfies
 $\ft{\mu} = \mu$ by Poisson's summation formula.
More generally, the measure $\delta_\Lam$ is crystalline
for any set $\Lam$  of the form \eqref{eqI2.1}, i.e.\ any set
$\Lam \sbt \R$ having a periodic structure.

There exist also examples of crystalline measures
on $\R$, whose support is \emph{not} contained in any set
$\Lam$ with  a periodic structure. Constructions
of such examples, using different approaches, were
given in the papers \cite{LO16}, \cite{Kol16}, 
\cite{Mey16}, \cite{Mey17}, \cite{RV19}.

Recently, new progress was achieved
by Kurasov and Sarnak \cite{KS20}
who constructed examples of crystalline measures
on $\R$ enjoying some remarkable
 properties,  answering several questions
left  open by the above mentioned papers.
To state the result, we recall the
terminology introduced in \secref{secTUB} above:

\begin{definition}
\label{defTG}
We say that a set $S \sbt \R$ has 
\emph{tempered growth} if there is $N$ such that
\begin{equation}
\label{eqI2.5}
\# (S \cap (-r,r)) = O(r^N), \quad r \to +\infty.
\end{equation}
\end{definition}
We note that if a set $S \sbt \R$ has  tempered growth then
the measure $\del_S$ is a tempered distribution,
 and also the converse is true.

\begin{thm}[{\cite{KS20}}]
\label{thmI1}
There exist
crystalline measures $\mu$ 
of the form \eqref{eqI1.10}
with the following properties:
\begin{enumerate-num}
\item \label{thmI1:i}
$\Lambda$ is a set of bounded density;
\item \label{thmI1:ii}
$a_\lam=1$ for all $\lam\in\Lam$,
i.e.\ we have $\mu = \del_\Lam$;
\item \label{thmI1:iii}
$\Lam$ does not have a periodic structure,
i.e.\ it is not of the form \eqref{eqI2.1};
\item \label{thmI1:iv}
$S$ is a set of tempered growth.
\end{enumerate-num}
\end{thm}

Actually, the crystalline measures
constructed in \cite{KS20} have even stronger
properties than stated above -- we only mentioned
the properties that will be used in this paper.
In the more recent works
\cite{Mey20}, \cite{OU20}, alternative approaches to
the construction of crystalline measures
with these properties can be found.

\subsection{Interpolation by functions with a sparse spectrum}
We now turn to the second
main ingredient in our proof of \thmref{thmA1}.
It can be described in the context of the classical
uncertainty principle, which states that a nonzero function
$f$ and its Fourier transform $\ft{f}$ cannot  be ``too small''
at the same time. This general principle  has 
many concrete versions, see  \cite{HJ94}. 

In particular, using complex analysis 
one can show that if $f \in L^1(\R)$
has compact support, and if $\ft{f}$ vanishes 
on a set  $S \sbt \R$ satisfying
\begin{equation}
\label{eqI1.6}
\limsup_{r \to +\infty} \frac{\# (S \cap (-r,r))}{r} = + \infty,
\end{equation}
then $f = 0$ a.e. This  fact was used in 
 \cite{LM91}, \cite{KL96} in order
to prove that if a nonzero function
$f \in L^1(\R)$ has  compact support,
and if  $f$ tiles at some level $w$
by a translation set $\Lambda$ 
of bounded density, then
$\Lambda$ must have a periodic structure
(\thmref{thmLM91}).

On the other hand, suppose that $\Om \sbt \R$ is a set 
which contains arbitrarily long intervals
(and so $\Om$ must be unbounded, 
but can be very sparse). Then given any discrete
set $S \sbt \R$ of tempered growth,
there exists a nonzero function  $f \in L^1(\R)$,
$\supp(f) \sbt \Om$, such that
$\ft{f}$ vanishes on $S$. This is a consequence
of the following result:

\begin{thm}
  \label{thmB2}
Let $\Om \sbt \R$ be
 a set which contains arbitrarily long intervals,
and let $S \sbt \R$ be  a set 
of  tempered growth. Then given any values
$\{c(s)\} \in \ell^1(S)$, one
can find  a smooth function 
$f \in L^1(\mathbb R)$,
$\supp(f) \sbt \Om$,
such that $\ft{f}(s)=c(s)$ for all  $s \in S$.
\end{thm}

If we call $\supp(f)$ the ``spectrum'' of the
function $\ft{f}$, then the result means that 
every discrete function in $\ell^1(S)$
can be interpolated
by a continuous function (the Fourier transform
of an $L^1$ function) with spectrum
in $\Om$. We refer
the reader to \cite{OU16} where the problem of
interpolation by functions with a given spectrum
is discussed in detail.

The approach that we use to prove  \thmref{thmB2} is
 inspired by \cite[Section 3]{OU09}.

\subsection{Proof of \thmref{thmB2}}
We begin with a simple lemma.

\begin{lem}
  \label{lemC1}
Let $S \sbt \R$ be a discrete set of  tempered growth, and
let $N$ be a sufficiently large number such that 
\eqref{eqI2.5} holds. Then for any $t \in \R \setminus S$ and any
$p>N$ we have
\begin{equation}
\label{eqC1.1}
\sum_{s \in S} |s - t|^{-p} < +\infty.
\end{equation}
\end{lem}

\begin{proof}
The condition \eqref{eqI2.5} remains valid 
if we replace $S$ with $S - t$, so we may assume
that $t=0$. The  function $n(r) := \# (S \cap (-r,r))$
 is then  a step function vanishing near the origin. For each $R>0$
which is not a jump discontinuity point of $n(r)$, we have
\begin{equation}
\label{eqC1.2}
\sum_{|s| \leq R} |s|^{-p} = \int_{0}^{R} \frac{dn(r)}{r^{p}} 
= \frac{n(R)}{R^p} +  p \int_{0}^{R} \frac{n(r)}{r^{p+1}} dr
\end{equation}
which follows from integration by parts. But as
$n(r)=O(r^N)$ and $p>N$,  the right-hand side of
\eqref{eqC1.2} remains bounded
as $R \to \infty$. The condition
\eqref{eqC1.1} is thus established.
\end{proof}

\begin{proof}[Proof of \thmref{thmB2}]
We suppose that  $\Om \sbt \R$ is a set which contains arbitrarily long intervals,
and that $S \sbt \R$ is a discrete set of  tempered growth. 
We also choose and fix some enumeration $\{s_j\}$ of the set $S$.

Let $\Phi$ be an infinitely smooth function on $\R$
 supported on the interval $[-1,1]$, and such that
$\int \Phi = 1$. For each $r>0$ we denote
$\Phi_r(x) := r^{-1} \Phi(r^{-1}x)$. Define
\[
f_j(x) := \Phi_{r_j} (x - \tau_j) \, e^{2 \pi i s_j x}
\]
where the numbers $r_j > 0$ and $\tau_j \in \R$ will
be determined later on. Then we have
\[
\ft{f}_j(t) = \ft{\Phi} (r_j(t-s_j)) \, e^{-2 \pi i \tau_j (t-s_j)}.
\]
In particular, $\ft{f}_j(s_j) =  \ft{\Phi} (0) = 1$.

Let $N$ be a sufficiently large number such that 
\eqref{eqI2.5} holds, and let $p > N$. Since
$\ft{\Phi}$ is a Schwartz function, there is a constant
$C=C(\Phi,p)>0$ such that 
$|\ft{\Phi}(t)| \leq C |t|^{-p}$ for every nonzero $t$.
For $j$ fixed, we have
\begin{equation}
\label{eqC1.4}
\sum_{k \neq j} |\ft{f}_j(s_k)| = 
\sum_{k \neq j} |\ft{\Phi}(r_j(s_k-s_j))| \leq
C r_j^{-p} \sum_{k \neq j} |s_k-s_j|^{-p}.
\end{equation}
If we use \lemref{lemC1} with the set $\{s_k : k \neq j\}$
and $t = s_j$, the lemma yields that the sum on the right-hand
side of  \eqref{eqC1.4} converges.
We thus see that given any $\eps > 0$, we can choose
$r_j = r_j(S,\Phi,p,\eps) >0$ large enough such that
the right-hand side of  \eqref{eqC1.4} does not exceed $\eps$.
It follows that we can choose the numbers $r_j$ in
such a way that, no matter how the numbers $\tau_j$ 
are chosen, we have
\begin{equation}
\label{eqC1.6}
M := \sup_{j}
\sum_{k \neq j} |\ft{f}_j(s_k)|  \leq \eps.
\end{equation}

To any sequence $\mathbf{b} = \{b_j\}$ belonging to 
$\ell^1$, we associate a corresponding sequence 
$\mathbf{c} = \{c_k\}$ defined by
\[
c_k = \sum_{j} \ft{f}_j(s_k) b_j = 
b_k + \sum_{j \neq k} \ft{f}_j(s_k) b_j.
\]
It follows from 
\eqref{eqC1.6} that the mapping $\mathbf{b} \mapsto \mathbf{c}$
defines a bounded linear operator $A: \ell^1 \to \ell^1$ such that
$\|A - I\| = M \leq \eps$, where $I$ is the identity operator.
If we choose $\eps < 1$ then this implies
that $A$ is invertible in $\ell^1$.

Now suppose that we are given a 
sequence $\mathbf{c} = \{c_k\} \in \ell^1$.
Let $\mathbf{b} = \{b_j\}  \in \ell^1$ be the solution
of the equation $A \mathbf{b} = \mathbf{c}$, and define
\begin{equation}
\label{eqC1.8}
f(x) := \sum_{j} b_j f_j(x).
\end{equation}
We observe that 
\[
\|f_j\|_{L^1(\R)} = \|\Phi\|_{L^1(\R)}, \quad \sum |b_j| < \infty,
\]
which implies that the series \eqref{eqC1.8} converges
in $L^1(\R)$. It follows that
\begin{equation}
\label{eqC1.9}
\ft{f}(t) = \sum_{j} b_j \ft{f}_j(t),
\end{equation}
where the series \eqref{eqC1.9} converges
uniformly on $\R$. In particular we have
\begin{equation}
\label{eqC1.10}
\ft{f}(s_k) = \sum_{j} b_j \ft{f}_j(s_k) = c_k
\end{equation}
for each $k$.

Finally we observe that $f_j$ is supported on
the interval $I_j := [\tau_j - r_j, \tau_j + r_j]$.
Since the  $r_j$'s were chosen in a way that
does not depend on the values of the $\tau_j$'s,
we can use
the assumption that
$\Om$  contains arbitrarily long intervals
in order  to choose each $\tau_j$ such that
the interval $I_j$ lies  in $\Om$.
Moreover, by choosing these intervals e.g.\ such
that $\dist(I_j, I_k) \geq 1$ $(j \neq k)$
this implies that
$\supp(f) \sbt \Om$ and ensures that $f$ is
infinitely smooth.
Thus $f$ has all the required properties,
and \thmref{thmB2} is proved.
\end{proof}


\subsection{Proof of \thmref{thmA1}}
Now we can finish the proof of
\thmref{thmA1} based on the two results
above, namely,  \thmref{thmI1} and  \thmref{thmB2}.

Let $\mu = \del_\Lam$
be the crystalline measure given
by \thmref{thmI1}, that is, $\Lambda \sbt \R$
 is a set of bounded density but not of
a periodic structure, such that
 $\ft{\del}_\Lam$ is a pure point measure,
$\ft{\del}_\Lam = \sum_{s \in S} b_s \del_s$,
where $S \sbt \R$
 is a set of tempered growth.

Since the set $S$ is discrete and closed, 
it follows that there is $a>0$ such that we have
$\hat\delta_\Lambda=c\cdot\delta_0$
in $(-a,a)$. Moreover,  we must have $c>0$
due to \lemref{lem5.2}. 
By rescaling the set $\Lam$ if needed, we may 
 assume that   $\hat\delta_\Lambda=\delta_0$ in $(-a,a)$.
In particular, 
this means that the set $S$ contains the origin.

Now suppose that we are given a scalar $w$, and a set
$\Om \sbt \R$ which contains arbitrarily long intervals.
Let $\xi$ be any real number, $\xi \notin S$,
and define 
$S' := S \cup \{\xi\}$. Then also the set $S'$
has tempered growth. We 
 prescribe the values $c(\xi) = 1$, $c(0) = w$, 
and $c(s) = 0$ for all $s \in S \setminus \{0\}$,
then these values belong to $\ell^1(S')$.
Hence using \thmref{thmB2} we can find a
smooth function  $f \in L^1(\mathbb R)$,
$\supp(f) \sbt \Om$,
such that $\ft{f}(s)=c(s)$ for all  $s \in S'$.
This means that we have $\ft{f}(0)=w$ and
$\ft{f}(s)=0$ for all $s \in S \setminus \{0\}$.
It also implies that $\ft{f}(\xi) \neq 0$,
which ensures that the function $f$ is nonzero.

We conclude that  $\ft{f}$ vanishes
on the set $\supp(\hat\delta_\Lambda) \setminus \{0\}$,
i.e.\ the condition \eqref{eqI3.1} is satisfied,
and moreover we have $\ft{\delta}_\Lam(\{0\}) \cdot \int f = w$.
Due to \thmref{thm4.5.2} this implies
that $f+\Lambda$ is a tiling at level $w$, and the
proof is thus complete.
\qed


\section{Tilings of unbounded density}
\label{secUB1}

In this section we construct
examples of tilings $f + \Lam$
such that the set $\Lam$ does
\emph{not} have bounded density
(\thmref{thmA3}).
We are not aware of any  previous
example of such a tiling  in the literature.
Moreover, in our example the set
$\Lam$ has tempered growth
(see \defref{defTG}) and the
function $f$ is in the Schwartz class.
We will also show
that in any such a tiling, the function $f$ must have
zero integral, no matter what  the 
level $w$ of the tiling is (\thmref{thmA9}).

\subsection{Translation sets with unbounded density}
We will construct  tilings $f + \Lam$ such that
the translation set $\Lam \sbt \R$  has the form
 \begin{equation}
\label{eqR7.30.12}
\Lam = \bigcup_{n \in \Z} \Lam_n,
\quad
\Lam_n \sbt (n-\eps, n+\eps), \quad
\# \Lam_n = n^2+1,
\end{equation}
where $\eps >0$ is an arbitrarily small number.
Then the condition \eqref{eqI10.1} is not satisfied
and $\Lam$ does not have bounded density.
On the other hand it follows from
\eqref{eqR7.30.12}  that $\#(\Lambda\cap(-r,r))  = O(r^3)$ 
as $r\to +\infty$, so
the set $\Lam$ has tempered growth and
$\delta_\Lambda$ is a tempered  distribution on $\R$.
We will prove the following theorem:

 \begin{thm}
\label{thmR7.43}
Given any $\eps >0$ there is a set
$\Lam \sbt \R$ of the form \eqref{eqR7.30.12},
such that 
 \begin{equation}
\label{eqR7.30.9}
\ft{\delta}_\Lam = \delta_0 - \frac{\delta''_0}{4\pi^2}
\quad \text{in $(-\half, \half)$.}
\end{equation}
 \end{thm}

In order to better understand what is behind 
condition \eqref{eqR7.30.9}, consider
the measure $\mu$ on $\R$ which assigns the mass $n^2+1$ to
each point $n \in \Z$. Using Poisson's summation formula
one  can check that the Fourier transform 
of $\mu$ is given by 
$\ft{\mu} = \sum_{k\in \Z} (\delta_k - (4\pi^2)^{-1} \delta''_k)$.
\thmref{thmR7.43} now says that  one can
``redistribute'' the mass $n^2+1$ assigned at each
point $n \in \Z$ into equal  unit masses at $n^2+1$ distinct
 points of a set $\Lam_n$ contained 
in a small neighborhood of $n$,
in such a way that the Fourier transform
of the new measure $\del_\Lam$ thus obtained 
remains unchanged  on the interval $(-\half, \half)$.

The role of condition \eqref{eqR7.30.9} in the tiling problem
is clarified by the following lemma:

 \begin{lem}
\label{lemR7.30}
Let $\Lam \sbt \R$ be a set of
tempered growth, and suppose that 
there is $a>0$ such that 
\begin{equation}
\label{eqR7.33}
\supp(\hat\delta_\Lambda) \cap (-a,a) =  \{0\}.
\end{equation}
Then given any scalar $w$ one can find
a nonzero Schwartz function $f$,
 such that $f+\Lambda$ is a tiling at level $w$.
 \end{lem}

\begin{proof}
It is well-known that a distribution supported by the origin is a finite linear combination of derivatives
of $\delta_0$ (see \cite[Theorem 6.25]{Rud91}). Hence
\eqref{eqR7.33} implies that
\begin{equation}
\label{eqR7.41}
\ft{\delta}_\Lam = \sum_{j=0}^{n} c_j \delta_0^{(j)}
\quad \text{in $(-a, a),$}
\end{equation}
and we may assume
that the highest order coefficient $c_n$ is nonzero.
It follows that
\begin{equation}
\label{eqR7.44}
t^n \, \ft{\del}_\Lam(t) = c_n n! (-1)^n  \cdot \del_0(t)
\quad \text{in $(-a, a)$}
\end{equation}
(see e.g.\ \cite[Section 9.1, Exercise 3]{Fol99}).

Let $\varphi$ be a nonzero Schwartz function,
$\supp(\ft{\varphi})
\sbt (-a, a)$, $\ft{\varphi}(0)=w$. Then 
\begin{equation}
\label{eqR7.42}
f(x) := \frac{\varphi^{(n)}(x)}{c_n n!(-2\pi i)^n}
\end{equation}
is also a nonzero Schwartz function.
We claim that $f+ \Lam$ is a tiling at level $w$. 
Indeed, 
\begin{equation}
\label{eqR7.46}
\ft{f}(t) \, \ft{\delta}_\Lam(t) = 
 \frac{\ft{\varphi}(t)  t^n}{c_n n! (-1)^n} 
\cdot \ft{\delta}_\Lam(t) = 
\ft{\varphi}(t) \, \delta_0(t) = w \cdot \delta_0(t),
\end{equation}
where in the second equality we used
\eqref{eqR7.44}. This implies that
 $f \ast \delta_\Lam=w$ as needed.
\end{proof}

Since condition \eqref{eqR7.30.9} implies \eqref{eqR7.33},
we see that \thmref{thmA3} is a consequence
of \thmref{thmR7.43} and \lemref{lemR7.30}.
It therefore remains to prove \thmref{thmR7.43}.

\subsection{Kargaev's implicit function method}
In order to prove
\thmref{thmR7.43} we will use a variant of Kargaev's 
implicit function method \cite{Kar82}.
See also \cite{KL16}, \cite{Lev20} 
for applications of the method in the
construction of translational tiling examples.

The proof will be done in several steps.

\subsubsection{} 
For each $k=1,2,3,\dots$ let
 $\chi_k$ be the function on $\R$ defined by
 \begin{equation}
\label{eqR20.1}
\chi_k(x) = k-j+1, \quad
x \in \big[\tfrac{2(j-1)}{k(k+1)}, \,
 \tfrac{2j}{k(k+1)}\big), \quad 1 \leq j \leq k,
\end{equation}
and $\chi_k(x) = 0$ outside the interval
$[0, \frac{2}{k+1})$.

 Let $\{\alpha_n\}$, $n\in\Z$, be a bounded sequence of real numbers. To such a sequence
 we associate a function $F$ on the real line, defined by
 \begin{equation}
\label{eqR21.4}
F(x):=\sum_{n \in \Z} F_n(x), \quad x\in\R,
 \end{equation}
where we let
 \begin{equation}
\label{eqR21.5}
F_n(x) := \sign(\alpha_n) \cdot \chi_{n^2+1}(\tfrac{x-n}{\alpha_n})
 \end{equation}
for each $n \in \Z$ such that $\alpha_n \neq 0$, while if $\alpha_n=0$
then we let $F_n(x):=0$.

(The notation 
 $\sign(\alpha_n)$ means $+1$ or $-1$ depending on whether
$\alpha_n>0$ or $\alpha_n<0$.)

One can verify, using the assumption
that  the sequence $\{\alpha_n\}$ is bounded, 
that the series
\eqref{eqR21.4} converges in the space
of tempered distributions. For instance,
this follows from the fact 
 that $(1+x^2)^{-1} \sum |F_n(x)|$
is a bounded function.

 \begin{thm}
\label{thmR7.2}
Let $0<\eps<1$. Then for any
sufficiently small $r>0$ one can find
a real sequence $\alpha =
\{\alpha_n\}$, $n\in\Z$, satisfying
\begin{equation}
\label{eqR7.2.5}
(1-\eps) r \le |\alpha_n|  \le   (1+\eps) r,
\quad n \in \Z,
\end{equation}
and such that $\ft{F}=0$ in $(-\half,\half)$,
where $F$ is the function  defined by \eqref{eqR21.4}.
 \end{thm}

We will first prove \thmref{thmR7.2} and then
use it to deduce  \thmref{thmR7.43}.

\subsubsection{}
We will need the following lemma.
\begin{lem}
  \label{lemR20.90}
The function $\chi_k$ has the following properties:
\begin{enumerate-roman}
\item \label{lemR20.90.1} $\int \chi_k(x) dx = 1$;
\item \label{lemR20.90.2} $\int x \chi_k(x) dx \leq C k^{-1}$;
\item \label{lemR20.90.3} $| \ft{\chi}_k(-s) - 1 | \leq C |s| k^{-1}$;
\item \label{lemR20.90.4} 
$\big|  v \big( \ft{\chi}_{k}(- v) - 1\big) -
u \big( \ft{\chi}_{k}(- u) - 1\big) \big|
\leq C k^{-1} |v-u| \cdot \max\{|u|, |v|\}$
\end{enumerate-roman}
for every $s,u,v \in \R$,  where $C>0$ is an absolute constant.
\end{lem}

\begin{proof}
The property \ref{lemR20.90.1} can be checked
directly. In order to establish \ref{lemR20.90.2}
we can estimate the left hand side by
$k \int_0^{2/(k+1)} x dx = 2k/(k+1)^2$.
Together with the estimate
\begin{equation}
\label{eqR20.4.2}
\textstyle | \ft{\chi}_k(-s) - 1 | = 
\Big| \int  \chi_k(x) (e^{2 \pi i sx}-1) dx \Big|
\leq 2 \pi |s| \int  x \chi_k(x) dx
\end{equation}
this yields  \ref{lemR20.90.3}. 
Finally, to establish \ref{lemR20.90.4}
consider the function
 $\varphi(x) := x ( e^{2 \pi i x} - 1)$.
We may suppose that $u<v$,  then
the left hand side of \ref{lemR20.90.4} is equal to
\begin{align}
\label{eqR20.5.7}
& \Big|  \int \chi_k(x) \frac{\varphi(vx) - \varphi(ux)}{x} dx \Big| =
\Big|  \int \chi_k(x) \int_u^v \varphi'(sx) ds dx \Big|  \\
\label{eqR20.5.8}
&\qquad\qquad\leq |v-u| \cdot \max_{s \in [u,v]} 
  \int \chi_k(x) |\varphi'(sx)|  dx.
\end{align}
If we use the estimate $|\varphi'(sx)| \leq C|sx|$, then
\eqref{eqR20.5.8} and  \ref{lemR20.90.2} imply
\ref{lemR20.90.4}.
\end{proof}

\subsubsection{}
Let $X$ be the space of all bounded sequences of real numbers
$\alpha = \{\alpha_n\}$,  $n \in \Z$,
endowed with the norm
\[
\|\alpha\|_X := \sup_{n \in \Z} |\alpha_n|
\]
that makes $X$ into a real Banach space.

Denote $I := [-\half, \half]$. Let
 $Y$ be the space of all continuous functions $\psi:I \to \CC$ 
satisfying $\psi(-t)=\overline{\psi(t)}$ for all $t\in I$. 
If we endow $Y$ with the  norm $\|\psi\|_Y=\sup|\psi(t)|$, $t\in I$,
then also $Y$ is a real Banach space.

Let $\alpha = \{\alpha_n\}$, $n \in \Z$, be a  sequence  in $X$. Define 
\begin{equation}
\label{eqR2.3}
(R \alpha)(t):= \sum_{n \in \Z}  e^{2\pi int}\cdot \alpha_n 
\big( \ft{\chi}_{n^2+1}(-\alpha_n t) - 1\big).
\end{equation}
We notice that the terms of the series \eqref{eqR2.3} 
are elements of $Y$.
By \lemref{lemR20.90}\ref{lemR20.90.3}, the $n$'th term of the series
is bounded by 
$C |\alpha_n|^2  (n^2+1)^{-1} $
uniformly on $I$, 
 where $C>0$ does not
 depend on $\alpha$ or $n$.
Hence the series \eqref{eqR2.3} converges
uniformly on $I$ to an element of $Y$, and we have
\begin{equation}
\label{eqR2.4.1}
\|R \alpha\|_Y \leq C \|\alpha\|^2_X,
\end{equation}
where the constant $C$ does not depend on $\alpha$.

We note that the mapping $R : X \to Y$ 
defined by \eqref{eqR2.3} is  \emph{nonlinear}.

\subsubsection{}
For each $r>0$ let
 $U_r$ denote the closed ball of radius $r$ around the origin
in $X$:
\begin{equation}
\label{eqR3.1.1}
U_r := \{\alpha \in X : \|\alpha\|_X \leq r\}.
\end{equation}

\begin{lem}
\label{lemR3.1}
Given any $\rho > 0$ there is $r>0$ such that
\begin{equation}
\label{eqR3.1.2}
\|R\beta - R\alpha\|_Y \leq \rho \|\beta-\alpha\|_X,
\quad \alpha,\beta \in U_r.
\end{equation}
In particular, if $r$ is small enough then $R$ is
a contractive (nonlinear) mapping on $U_r$.
\end{lem}

\begin{proof}
Let $\alpha,\beta \in U_r$. Then using
\eqref{eqR2.3}  we  have
\begin{equation}
\label{eqR3.2}
(R \beta - R \alpha)(t) =
\sum_{n \in \Z} e^{2\pi int}\cdot 
\Big[
 \beta_n \big( \ft{\chi}_{n^2+1}(-\beta_n t) - 1\big) -
 \alpha_n  \big( \ft{\chi}_{n^2+1}(-\alpha_n t) - 1\big) \Big].
\end{equation}
By \lemref{lemR20.90}\ref{lemR20.90.4}, 
 the $n$'th term of the series  is bounded by 
$C r (n^2+1)^{-1} |\beta_n - \alpha_n|$
uniformly on $I$, 
 where $C>0$ does not
 depend on $r$, $\alpha$, $\beta$ or $n$.
Hence the series converges
uniformly on $I$  and 
$\|R\beta - R\alpha\|_Y \leq Cr  \|\beta-\alpha\|_X$,
where $C$ is a constant not depending on $r$, $\alpha$ or $\beta$.
It thus suffices to choose $r$ small enough so that  $C r \leq  \rho$.
\end{proof}

\subsubsection{}
For each element $\psi \in Y$ we denote  by
$\F(\psi)$ the sequence
\begin{equation}
\label{eqR1.1.10}
\ft{\psi}(n) = \int_{I} \psi(t) e^{-2\pi i n t} dt, \quad n \in \Z,
\end{equation}
of Fourier coefficients of $\psi$. 
Since the Fourier coefficients
$\ft{\psi}(n)$ are
real and bounded, we have
a linear mapping  $\F: Y \to X$ satisfying
$\|\F(\psi)\|_X \leq \|\psi\|_Y$.

\begin{lem}
\label{lemR4.2}
Given any $\eps>0$ there is 
$\delta>0$ with the following property: 
Let $\beta \in X$,
$\|\beta\|_X  \le \delta$. Then one can find an element
 $\alpha \in X$,
$\|\alpha - \beta\|_X  \le   \eps \|\beta\|_X$,  which solves the
equation  $\alpha + \F(R \alpha) = \beta$.
\end{lem}

 \begin{proof} 
Fix $\beta \in X$ such that  $\|\beta\|_X \le \delta$,
and let 
\[
B = B(\beta,\eps) :=\{ \alpha \in X : \|\alpha-\beta\|_X \le \eps \|\beta\|_X\}.
\]
We observe that if $\alpha \in B$ then  $\|\alpha\|_X \le (1+\eps) \|\beta\|_X$.
Define a map $H: B \to X$ by
\[
H(\alpha) := \beta - \F(R\alpha), \quad \alpha \in B,
\]
and notice that an element $\alpha \in B$ is a solution to the equation
$\alpha + \F(R\alpha) = \beta$ 
if and only if $\alpha$ is a fixed point of the map $H$.
 
Let us show that if $\delta$ is small enough then $H(B)\subset B$.
 Indeed, if $\alpha \in B$ then using \eqref{eqR2.4.1} we have
\[
\|H(\alpha)-\beta\|_X = \|\F(R\alpha)\|_X 
\leq \|R \alpha\|_Y  \leq
 C \|\alpha\|^2_X \leq C (1+\eps)^2 \|\beta\|^2_X.
\]
Hence if we choose $\delta$ such that 
$C(1+\eps)^2 \delta \leq \eps$ then
we obtain
\[
\|H(\alpha)-\beta\|_X \leq \eps \|\beta\|_X,
\]
and it follows that $H(B)\subset B$.
 
It also follows  from \lemref{lemR3.1} 
that if $\delta$ is small enough, then $H$ is a contractive 
mapping from the closed set $B$ into itself. Indeed,
let $\alpha',  \alpha'' \in B$, then we have
\[
\|H(\alpha'') - H(\alpha') \|_X = 
\|\F(R\alpha'' - R\alpha') \|_X \le
\|R\alpha'' - R\alpha' \|_Y \le
\rho \|\alpha'' - \alpha'\|_X,
\]
where $0<\rho<1$.
Then the Banach fixed point theorem implies that
$H$ has a (unique) fixed point $\alpha \in B$,
which yields the desired solution.
 \end{proof}

\subsubsection{Proof of \thmref{thmR7.2}}
Let $r>0$, and let $\beta = \{\beta_n\}$, $n \in \Z$, be a real sequence 
defined by $\beta_n := (-1)^n r$. Then $\|\beta\|_X = r$. 
 If  $r=r(\eps)>0$ is small enough,
then by  \lemref{lemR4.2} there is an element
$\alpha \in X$,  
$\|\alpha - \beta\|_X  \le   \eps \|\beta\|_X$,  which solves the
equation  $\alpha + \F(R \alpha) = \beta$.

We observe that the estimate
$\|\alpha - \beta\|_X  \le   \eps \|\beta\|_X$
implies \eqref{eqR7.2.5}.

The relation $\alpha + \F(R \alpha) = \beta$ means
that $\beta - \alpha$ is the
 sequence of Fourier coefficients  of 
the function $R \alpha$. This implies
that the series
\[
\sum_{n \in \Z} (\beta_n - \alpha_n) e^{2\pi int} 
\]
converges in $ L^2(I)$  to $R \alpha$.

Since we have $\beta_n = (-1)^n r$, $n \in \Z$,
the series
\[
\sum_{n \in \Z} \beta_n e^{2\pi int} 
\]
converges in the distributional sense to  the measure
$r \cdot \del_{\Z +  \frac1{2}}$  on $\R$.
In particular, this series converges to zero
 in the open interval $(-\frac1{2},\frac1{2})$.

Let  $F$ be the function given by \eqref{eqR21.4}
associated to the sequence $\alpha=\{\alpha_n\}$. Then
\[
\hat F (-t) = \lim_{N\to\infty}\sum_{|n|\le N}\hat F_n(-t)
\]
 in the sense of distributions, and by \eqref{eqR21.5} we have
 \begin{equation}
\label{eqR21.6}
\ft{F}_n(-t) = e^{2 \pi i n t} \cdot \alpha_n \, \ft{\chi}_{n^2+1}(-\alpha_n t).
 \end{equation}
Hence
\[
\hat F(-t) =
\lim_{N\to\infty} \Big[ \sum_{|n|\le N} \beta_n e^{2\pi int} 
 - \sum_{|n|\le N} (\beta_n - \alpha_n )e^{2\pi int} 
 + \sum_{|n|\le N}   e^{2\pi int}\cdot \alpha_n 
\big( \ft{\chi}_{n^2+1}(-\alpha_n t) - 1\big) \, \Big].
\]
The first sum converges 
in the distributional sense to zero in $(-\half, \half)$.
The second sum converges in
 $L^2(I)$ to $R \alpha$. The third
sum converges  to $R\alpha$
uniformly on $I$.
 It follows that the distribution $\hat F$ vanishes
 in the open interval $(-\half, \half)$.
\qed

\subsection{Proof of \thmref{thmR7.43}}
Let $0<\eps<1$ be given, and for $r=r(\eps)>0$ sufficiently small 
let  $\{\alpha(n)\}$, $n \in \Z$,
 be the sequence given by \thmref{thmR7.2}. Define
 \begin{equation}
\label{eqR7.40.1}
\Lam_n = \Big\{ n +  \frac{2 j\alpha_n }{(n^2+1)(n^2+2)} , \quad 1 \leq j \leq n^2+1  \Big\}
\end{equation}
and
 \begin{equation}
\label{eqR7.40.2}
\Lam = \bigcup_{n \in \Z} \Lam_n.
\end{equation}
We have $\alpha_n \neq 0$
for every $n \in \Z$,
due to \eqref{eqR7.2.5}. Hence $\Lam_n$ is a 
 set with exactly $n^2+1$ elements.
The set $\Lam_n$ is contained in the
interval $[n - |\alpha_n|, n + |\alpha_n|]$.
This yields \eqref{eqR7.30.12}  provided
that $r>0$ is small enough,
again due to \eqref{eqR7.2.5}.
In particular we may assume that
the sets $\Lam_n$ are pairwise disjoint.

Observe that the distributional derivative of 
the function $F$ in \eqref{eqR21.4} is
\[
F'=\sum_{n\in \Z} F'_n =
\sum_{n\in \Z}((n^2+1) \delta_n - \delta_{\Lam_n}),
\]
and hence
\[
\delta_{\Lam} = \sum_{n\in \Z} (n^2+1) \delta_n -  F'.
\]
The Fourier transform
of the measure $\sum_{n \in \Z} (n^2+1) \delta_n$ is
$\sum_{k \in \Z} \big(\delta_k - (4\pi^2)^{-1}  \delta''_k \big)$,
which follows from Poisson's summation formula.
This implies that
\[
\hat \delta_\Lambda= \delta_0 - \frac{\delta''_0}{4\pi^2}  -\ft{{F'}}
\quad \text{in $(-\half, \half)$.}
\]
But since $\ft{F}$ vanishes in $(-\half,\half)$, then the same is true 
for $\ft{{F'}}$, so \eqref{eqR7.30.9} is established.
\qed

\subsection{Proof of \thmref{thmA9}}
Finally we show that if
$f+\Lam$ is a tiling at some level $w$,
where $\Lambda \sbt \R$ is  any set of
tempered growth but not of bounded density,
and $f$ is any function in the Schwartz class,
then $f$ must have zero integral.

Suppose to the contrary that 
$\int f = \ft{f}(0)$ is nonzero. Then,
due to the continuity and smoothness of $\ft{f}$,
there is a Schwartz function
$g$ such that $\ft{f} \cdot \ft{g}=1$
 in some neighborhood $(-a,a)$ of the origin.
Let $h>0$ be a Schwartz function
with $\supp(\ft{h})\subset (-a,a)$, then 
\[
\ft{h} \cdot \ft{g} \cdot \ft{f}  =\ft{h}
\]
 and hence
\[
h \ast g \ast f = h.
\]
It follows that
\[
h \ast \del_\Lam =
(h \ast g \ast f) \ast \del_\Lam =
(h \ast g) \ast (f \ast \del_\Lam) =
w \cdot \textstyle\int  (h \ast g),
\]
where in the last equality we used the 
tiling assumption $f \ast \del_\Lam = w$ 
(the associativity of the convolution is justified
since $\del_\Lam$ is a tempered distribution
and $f,g,h$ are Schwartz functions,
see \cite[Theorem 7.19]{Rud91}).

We conclude that $h + \Lam$ is a tiling
(at a certain level). But it is known, see
\cite[Lemma 2.1]{KL96},
that if $h$ is a \emph{nonnegative}, nonzero
function and if $h + \Lam$ is a tiling,
then the set $\Lam$ must have bounded density.
We thus arrive at a contradiction.
\qed


\section{Open problems} 
\label{secOP1}

We conclude the paper by posing some open problems.

\subsection{Tiling  at level zero}
The following problem was already mentioned 
in \secref{secTLZ} above:
Let $f+\Lam$ be a tiling  at level zero,
where the function $f \in L^1(\R)$ is nonzero and
the set $\Lam \sbt \R$ is  nonempty and has bounded density.
Does it follow that $\Lam$ has a \emph{uniform density}
$D(\Lam)$\,?

In Fourier analytic terms,
the problem can be equivalently  stated 
as follows: Let  $\Lam \sbt \R$ be
  nonempty and have bounded density, and suppose
that $\ft{\del}_\Lam$ vanishes on some
open  interval $(a,b)$. Does 
$\Lam$ necessarily have a uniform density?

What makes the problem nontrivial is 
the  existence of tilings  $f+\Lam$ 
at level zero such that
$\ft{\del}_\Lam$ is not
a scalar multiple of $\del_0$ in any
neighborhood of the origin,
see \cite[Section 5]{Lev20}.
In particular, \lemref{lem5.2} does not apply.

We note that by \thmref{thmA6} the set
$\Lam$ must be relatively dense,
so if the density $D(\Lam)$ exists then it 
is a strictly positive number.

\subsection{Non-periodic tilings} 
Let $f$ be a nonzero function in $L^1(\R)$, and suppose that
the set $\{x  : f(x) \neq 0\}$ has \emph{finite measure}.
If $f$ tiles at some level $w$ by a translation
set $\Lam \sbt \R$ of  bounded density, does it follows
that $\Lam$ has a periodic structure?

\thmref{thmLM91} does not apply here,
since $f$ is \emph{not} assumed to have  compact support.

 Does there exist a measurable set
$\Om \sbt \R$, $0<\mes(\Om)<+\infty$,
whose indicator function $\1_\Om$ can tile
at level one, or, a weaker requirement, at some other integer
level $w$, with a translation set $\Lam \sbt \R$ 
that \emph{does not} have a periodic structure?

Notice that
such a set $\Om$ (if it exists) must be unbounded, 
again due to \thmref{thmLM91}.

\subsection{Tilings of unbounded density}
Let $f \in L^1(\R)$ be nonzero and have \emph{compact support}, and 
suppose that $f+\Lam$ is a tiling at some level $w$, where $\Lam \sbt \R$
 is a discrete set (not a multi-set) of tempered growth. 
Does it follow that $\Lam$ is of the form \eqref{eqI2.1},
i.e. $\Lambda$ is a set
of bounded density having a periodic structure?

In other words, the question is whether
\thmref{thmLM91} remains valid if the set
$\Lambda$  is not assumed to have
 bounded density, but only tempered growth.

We note that \thmref{thmA3} does \emph{not} provide a negative
 answer to this question, since the function $f$ constructed in 
the proof of this theorem has \emph{unbounded} support.

\subsection{Lattice tilings} 
Let $f \in L^1(\R^d)$, $d \geq 1$, and suppose that
$f$ tiles at some level $w$
with a translation set $\Lambda \sbt \R^d$ of
 bounded density.\footnote{A set
$\Lambda \sbt \R^d$ is said to have
 \emph{bounded density} if there exists
$M>0$ such that $\#(\Lam \cap (x+B)) \leq M$
for all $x \in \R^d$, where $B$ is the open
unit ball in $\R^d$.}
 Does there necessarily exist a \emph{lattice}
$L \sbt \R^d$ such that
$f+L$ is also a  tiling,  possibly at a 
different level $w'$\,?

The answer is  known to be affirmative 
in the special case where $\Lambda$ is assumed to be
a disjoint union of finitely
many  translated lattices, namely,
$\Lam = \biguplus_{j=1}^{N} (L_j + \tau_j)$
where each $L_j$ is a lattice in
$\R^d$ and the $\tau_j$ are translation vectors.
This was proved in dimension one in 
\cite[p.\ 673]{KL96}, while in several dimensions
the result was proved more recently in \cite[Theorem 1.6]{Liu18}.
In both proofs, number theory plays an essential
 role:  the proof in $\R$ uses
the classical Skolem--Mahler--Lech theorem,
while in $\R^d$ the proof relies on
a result due to Evertse,  Schlickewei and  Schmidt
\cite{ESS02}.



\bibliographystyle{amsplain}



\begin{dajauthors}
\begin{authorinfo}[kolount]
M. N. Kolountzakis\\
Department of Mathematics and Applied Mathematics\\
University of Crete\\
Voutes Campus, GR-700 13\\
Heraklion, Crete\\
Greece\\
kolount@uoc.gr\\
\url{http://mk.eigen-space.org/}
\end{authorinfo}
\begin{authorinfo}[lev]
Nir Lev\\
Department of Mathematics\\
Bar-Ilan University\\
Ramat-Gan 5290002\\
Israel\\
levnir@math.biu.ac.il\\
\url{https://u.math.biu.ac.il/~levnir/}
\end{authorinfo}
\end{dajauthors}

\end{document}